\documentclass[11pt]{article}

 \usepackage{array}
 \usepackage{enumitem}
\newcolumntype{C}[1]{>{\centering\arraybackslash}m{#1}}
\usepackage{subfig}
\usepackage{graphicx}
\usepackage{amssymb}
\usepackage{mathtools}
\usepackage{amsmath}
\usepackage{amsthm}
\usepackage{latexsym}
\usepackage{amsfonts}
\usepackage{etoolbox}
\usepackage{multirow}
\usepackage{float}
\usepackage{amsmath,color}
\usepackage{url}
\usepackage{hyperref}
\usepackage[T1]{fontenc}

\usepackage{arydshln}
\usepackage{lipsum}

\newtheorem{theorem}{Theorem}{}
\newtheorem{lemma}{Lemma}{}
{}
\newtheorem{corollary}{Corollary}{}
{}
{}
{}

\theoremstyle{definition}
\usepackage[utf8]{inputenc}
\usepackage{authblk}

\newcommand{\tb}{\textcolor{blue}}

\title{\bf Characterizing the positive inertia index of connected signed graphs
in terms of girth}
\author{Suliman Khan$^{1}$, Sakander Hayat$^{2,}\footnote{Corresponding author's}$\,, Mohammed J.F. Alenazi$^{3}$}
\affil{$^{1}$Department of Mathematics and Physics,\\
University of Campania ``Luigi Vanvitelli'',\\
Viale Lincoln 5, Caserta, I–81100, Italy \vspace{0.2cm}\\
$^{2}$Faculty of Science,\\
Univeriti Brunei Darussalam,\\
Jln Tungku Link, Gadong BE1410,\\
Brunei Darussalam.\vspace{0.2cm}\\
$^{3}$Department of Computer Engineering,\\
College of Computer and Information Sciences (CCIS),\\
King Saud University, Riyadh 11451, Saudi Arabia\vspace{0.3cm}\\

 \small{Emails: suliman5344@gmail.com~(S.K.), ~ sakander1566@gmail.com~(S.H.)\\
 mjalenazi@ksu.edu.sa~(M.J.F.A.)}}

\date{}
\begin{document}

\maketitle

\vspace{-0.1cm}

\begin{abstract}  \noindent Let $G^\sigma=(G,\sigma)$ be a connected signed graph and 
$A(G^\sigma)$ be its adjacency matrix. The positive inertia index of $G^\sigma$, denoted 
by $p^{+}(G^\sigma)$, is defined as the number of positive eigenvalues of $A(G^\sigma)$. 
Assume that $G^\sigma$ contains at least one cycle, and let $g_{r}$ be its girth. In 
this paper, we prove $p^{+}(G^\sigma) \geq \lceil \frac {g_{r}}{2} \rceil-1$ for a signed 
graph $G^\sigma$. The extremal signed graphs corresponding to 
$p^{+}(G^\sigma) = \lceil \frac {g_{r}}{2} \rceil-1$ and $p^{+}(G^\sigma) =\lceil \frac {g_{r}}{2} \rceil$ 
are characterized, respectively. The results presented in this article extend the recent work on ordinary 
graphs by Duan and Yang (Linear Algebra Appl., 2024) to the context of signed graphs.
\end{abstract}
\vspace{0.2cm}
\\

\section{Introduction}
\indent
All graphs considered in this paper are assumed to be connected, simple, and undirected. Let \( G \) be 
a simple graph of order \( n \) with vertex set \( V(G) = \{ v_1, v_2, \dots, v_n \} \) and edge set 
\( E(G) \). The adjacency matrix \( A(G) \) of a graph \( G \) of order \( n \) is a symmetric 
\( n \times n \) matrix defined as: \( A(G)=(a_{ij}) \), where $a_{ij} =1$ if $v_i$ and $v_j$ are adjacent 
and $0$, otherwise. A signed graph \( G^\sigma = (G, \sigma) \) consists of a simple graph \( G = (V, E) \), 
referred to as its underlying graph, along with a sign function \( \sigma: E \to \{ +, - \} \). In signed graphs, 
edge labels are conventionally represented as \( \pm1 \). Accordingly, the adjacency matrix is intuitively 
constructed by assigning \( +1 \) or \( -1 \) to each edge based on its respective sign. Formally, the 
adjacency matrix of \( G^\sigma \) is denoted by \( A(G^\sigma)= a^\sigma_{ij}\) and defined as follows:

$$a^\sigma_{ij}=\left \{ \begin{array}{rcl}
\sigma(v_iv_j) & \mbox{if $v_i$ is adjacent to $v_j$}, \\
0 & \mbox{otherwise.}
\end{array} \right.$$\\

Since $A(G^\sigma)$ is a real symmetric matrix of a signed graph $G^\sigma$ and therefore it has real 
eigenvalues ordered by \( \lambda_1 \geq \lambda_2 \geq \cdots \geq \lambda_n \), which are also referred 
to as the eigenvalues of \( G^\sigma \). The number of positive, negative, and zero eigenvalues of 
\( A(G^\sigma) \) are known as the positive inertia index $p^+(G^\sigma)$, negative inertia index $n^-(G^\sigma)$, 
and nullity $\eta(G^\sigma)$ of \( G^\sigma \), respectively. Let $C^\sigma$ be a signed cycle. The sign of 
$C^\sigma$ is defined as the product of the signs of all its edges, denoted by $\sigma(C^\sigma)$. A signed 
cycle $C^\sigma$ is said to positive (resp. negative) if $\sigma(C^\sigma)=+$ (resp. $\sigma(C^\sigma)=-$). A 
signed graph $G^\sigma$ is said to be balanced if all the cycles in $G^\sigma$ are positive, or equivalently, 
if every cycle contains an even number of negative edges. Otherwise, $G^\sigma$ is called an unbalanced signed graph. 
We denote by \( P^\sigma_n \), \( C^\sigma_n \), \( K^\sigma_n \), and \( K^\sigma_{1,n} \) a signed path, a 
signed cycle, a signed complete graph, and a signed star, respectively, all of which are signed graphs on \( n \) 
vertices. The star \( K^\sigma_{1,n} \) is sometimes written as \( S^\sigma_{n+1} \). The notation 
\( K^\sigma_{n_1, n_2, \dots, n_t} \) represents a complete \( t \)-partite signed graph with part sizes 
\( n_1, n_2, \dots, n_t \). A signed graph (or simply a graph) is called trivial if it consists of a single 
vertex with no edges. A pendant vertex (or leaf) is a vertex with exactly one neighbor (or simply a vertex of 
degree 1). A pendant star is a subgraph consisting of a central vertex connected to one or more pendant vertices.\\

Suppose that a graph \( G \) contains at least one cycle. The girth of \( G \) is defined as the length of its shortest 
cycle and is denoted by \( g_r \). In \cite{Cheng2007}, Cheng and Liu showed that for an $n$-vertex graph \( G \), 
$\eta(G) \leq n-g_r+2$ if $4 \mid g_r$ and $\eta(G) \leq n-g_r$ if $4 \nmid g_r$. Connected graphs $G$ with 
$\eta(G) = n-g_r+2$ and $\eta(G) = n-g_r$ have been characterized by Zhou et al. \cite{Zhou2021}. 
Chang and Li \cite{Chang2022} characterized connected graphs $G$ with $\eta(G) = n-g_r-1$ and $\eta(G) = n-g_r+2$. 
Recently, Duan and Yang \cite{DuanYang2024} characterized positive inertia index of a graph $G$ in terms of its girth. In 
particular, they showed that $p^+(G) \geq \lceil \frac{g_r}{2}\rceil-1$. Furthermore, they 
characterized the corresponding extremal graphs with $p^+(G) = \lceil\frac{g_r}{2}\rceil-1$ and $p^+(G) = \lceil \frac{g_r}{2}\rceil$, 
respectively. Duan \cite{Duan2024} explored the relationship between the negative inertia index of a graph and its girth. 
His main result establishes a lower bound for the negative inertia index, given by 
$p^-(G) \geq \lceil \frac{g_r}{2}\rceil-1$. Some other finding about the spectral properties of a graph and its 
structural characteristics can be found in \cite{Ma2013,Oboudi2017,Oboudi2016,Torgasev1985}.\\

In recent studies, the inertia indices and nullity of signed graphs has been extensively explored. 
Yu et al. \cite{Yu2014} studied the positive inertia index of unicyclic signed graphs. Duan \cite{Duan2022} 
investigated both the negative and positive inertia indices of signed graphs in terms of their underlying graphs 
indices. Duan and Yang \cite{Duan2024} characterized signed graphs $G^\sigma$ with $n^-(G^\sigma)=1$ and triangle 
free signed graphs $G^\sigma$ with $n^-(G^\sigma)=2$, respectively. The nullity of unicyclic and bicyclic signed 
graphs has been studied by Fan et al. \cite{Fan2013} and Fan et al. \cite{Fan2014}, respectively. The nullity of signed graphs in terms of 
their girth has been extensively studied by Wu et al. \cite{Wu2022}. These results motivate us to explore bounds on the 
positive inertia index of connected signed graphs in terms of their girth. Next, we state our main results.

\begin{theorem}\label{Theorem1}
Let $G^\sigma=(G,\sigma)$ be a connected signed graph with girth $g_{r}$. Then $p^{+}(G^\sigma) \geq \lceil \frac {g_{r}}{2} \rceil-1$, 
where the equality holds if and only if:\\
(i) $G^\sigma=C^\sigma_{g_{r}}$ is a balanced signed cycle satisfying $g_{r} \equiv 0,1 \pmod{4}$ or an unbalanced signed cycle 
satisfying $g_{r} \equiv 2,3 \pmod{4}$.\\
(ii) $G^\sigma=K^\sigma_{n_1,n_2, \cdots , n_t}$ (where, $n_1,n_2, \cdots , n_t \geq 2$) is a balanced (with all positive cycles) 
complete multipartite signed graph.
\end{theorem}

By excluding the two cases where equality holds, we can establish tighter bounds for the positive inertia index of signed graphs, 
as presented in the following theorem.
\begin{theorem}\label{Theorem2}
Let $G^\sigma=(G,\sigma)$ be a connected signed graph with girth $g_{r}$. If $G^\sigma$ is neither a balanced cycle 
$C^\sigma_{g_{r}}$ with $g_{r} \equiv 0,1\pmod{4}$ or an unbalanced cycle with $g_{r} \equiv 2,3\pmod{4}$ nor a complete 
multipartite balanced signed graph, then $p^{+}(G^\sigma) \geq \lceil \frac {g_{r}}{2} \rceil$.
\end{theorem}

A signed graph $G^\sigma$ is called a \textit{canonical unicyclic graph} if it contains a single cycle and a pendant stars 
attached to none, some, or all of its vertices. In Figure \ref{canonical}, we illustrate a canonical unicyclic graph with 
girth $6$ and $p^+(G^\sigma) = \lceil \frac{6}{2} \rceil = 3$.

\begin{figure}[htbp!]
     \centering
  \includegraphics[width=5cm]{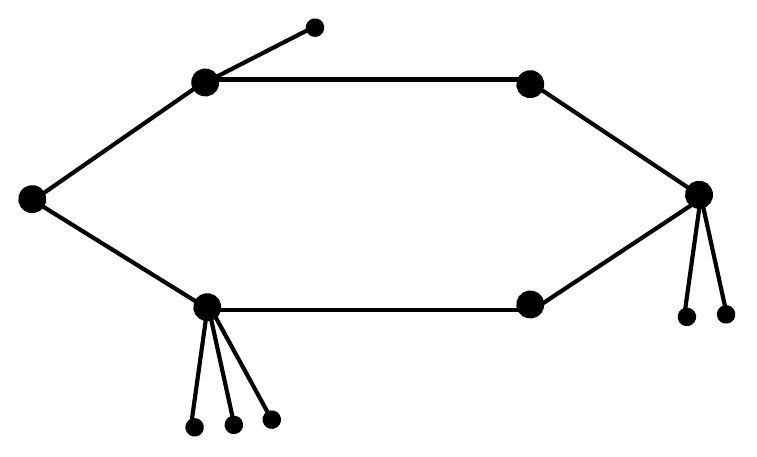}\\
  \caption{A canonical unicyclic graph with $p^+(G^\sigma) = \lceil \frac{6}{2} \rceil = 3$ having positive 
  eigenvalues $1.41421,~1.73205$, and $2.44949$.}\label{canonical}
\end{figure}

We define $B^\sigma(k,l,m)$ as a bicyclic signed graph obtained from connecting two distinct vertices by 
three internally vertex-disjoint paths, namely $P^\sigma_k$, $P^\sigma_l$, and $P^\sigma_m$. Clearly, any 
permutation of the parameters $k, l,$ and $m$ results in an isomorphic bicyclic signed graph. In the 
subsequent theorem, we characterize the extremal signed graphs corresponding to 
$p^+(G^\sigma) = \lceil \frac {g_{r}}{2} \rceil$.

\begin{theorem}\label{Theorem3}
Let $G^\sigma=(G,\sigma)$ be a connected signed graph with girth $g_{r}$. Then $p^{+}(G^\sigma) = \lceil \frac {g_{r}}{2} \rceil$ 
if and only if $G^\sigma$ is one of the following signed graph:\\
(i) $G^\sigma=C^\sigma_{g_{r}}$ is a balanced signed cycle satisfying $g_{r} \equiv 2,3\pmod{4}$ or an 
unbalanced signed cycle satisfying $g_{r} \equiv 0,1\pmod{4}$.\\
(ii) $G^\sigma$ is a canonical unicyclic signed graph where either exactly one vertex is attached to a pendant star, 
or multiple vertices are attached to a pendant star such that all paths having odd number of vertices between any two 
consecutive pendant stars when $g_r=1,3 \pmod4$.\\
(iii) $G^\sigma$ is a canonical unicyclic signed graph where either exactly one vertex is attached to a pendant star, 
or multiple vertices are each attached to a pendant star such that there is exactly one path with an even number of 
vertices (including zero) between any two consecutive pendant stars when $g_r=0,2 \pmod4$;\\
(iv) A signed graph $G^\sigma$ obtained from a balanced signed cycle $C^\sigma_{g_r}$ with $g_r \equiv 0,1 \pmod 4$ 
and signed star $S^\sigma_{t+1}$ by joining a vertex of $C^\sigma_{g_r}$ to the center of 
$S^\sigma_{t+1}$, where $t \geq 1$.\\
(v) A signed graph $G^\sigma$ obtained from an unbalanced signed cycle $C^\sigma_{g_r}$ with $g_r \equiv 2,3 \pmod 4$ 
and signed star $S^\sigma_{t+1}$ by joining a vertex of $C^\sigma_{g_r}$ to the center of 
$S^\sigma_{t+1}$, where $t \geq 1$.\\
(vi) Bicyclic signed graphs \( B^\sigma(5, 4, 5) \) and \( B^\sigma(5, 5, 5) \) with balanced cycles or a 
signed bicyclic graph \( B^\sigma(5, 3, 5) \) whose both cycles are negative, as depicted in Figure \ref{Bicyclic}.
\end{theorem}

The rest of the article is structured as follows:
In Section \ref{sec2}, we introduce preliminary lemmas, notations and fundamental results on the positive 
inertia index of a signed graph. Section \ref{sec3} is devoted to proving the main result, i.e., 
$p^{+}(G^\sigma) \geq \lceil \frac {g_{r}}{2} \rceil-1$, and characterizing the extremal signed graphs that 
satisfy $p^{+}(G^\sigma) = \lceil \frac {g_{r}}{2} \rceil-1$ and $p^{+}(G^\sigma) =\lceil \frac {g_{r}}{2} \rceil$. 
This section also completes the proofs of Theorems \ref{Theorem1}, \ref{Theorem2}, and \ref{Theorem3}.

\section{Preliminaries}\label{sec2}

This section presents some preliminary results required in subsequent sections.\\

The following is the so-called Cauchy Interlacing Theorem.
\begin{theorem}(\cite{Cvetkovic1980}): (Interlacing Theorem) Let \( P \) be a real \( n \times m \) 
matrix satisfying \( P^T P = I_m \), where \( m < n \). Suppose \( A \) is an \( n \times n \) real 
symmetric matrix with eigenvalues \( \lambda_1 \geq \lambda_2 \geq \dots \geq \lambda_n \). 
Define \( H = P^T A P \), which is an \( m \times m \) symmetric matrix with eigenvalues 
\( \nu_1 \geq \nu_2 \geq \dots \geq \nu_m \). Then, the eigenvalues of \( H \) interlace those of \( A \), 
meaning that
\[
\lambda_{n-m+j} \leq \nu_j \leq \lambda_j, \quad \text{for } j = 1, 2, \dots, m.
\]
\end{theorem}

The following lemma is a direct consequence of the Interlacing Theorem.
\begin{lemma}\label{Lemma01}
 Let $H^\sigma$ be an induced subgraph of a signed graph $G^\sigma$, then $p^{+}(H^\sigma) \leq p^{+}(G^\sigma)$.
\end{lemma}

The following lemma is derived from Theorem 1.1 part (b) by Gregory et al. \cite{Gregory2003}.
\begin{lemma}\label{Lemma001}
Let $G^{\sigma}$ be a signed graph and $x\in V(G^\sigma)$ be any pendant vertex in $G^\sigma$. 
Let $y \in V(G^\sigma)$ be a unique neighbor of $x$ in $G^\sigma$, then 
$p^+(G^{\sigma}) = p^+(G^{\sigma} - x - y) + 1$.
\end{lemma}

Next, we present a result by Schwenk \& Wilson \cite{Schwenk1978}.
\begin{lemma}\label{Lemma2}
Let $C^\sigma_n$ be a signed balanced cycle of order $n$. Then it has eigenvalues $2 \cos \frac{2\pi j}{n}$, where, $j=0,1,2, \cdots , n-1$.
\end{lemma}

Fan et al. \cite{Fan2013} showed the following lemma.
\begin{lemma}\label{Lemma1}
Let $C^\sigma_n$ be a signed unbalanced cycle of order $n$. Then, it has eigenvalues 
$2 \cos \frac{(2j-1)\pi}{n}$, where, $j=1,2, \cdots , n$.
\end{lemma}

From Lemmas \ref{Lemma2} and \ref{Lemma1}, we conclude the following:
\begin{lemma}\label{Lemma3}
Let $C^\sigma_n$ be a signed cycle of order $n$.\\
(i) If $C^\sigma_n$ is a balanced cycle, then
$$p^{+}(C^\sigma_n) =
\begin{cases}
 \lceil \frac {n}{2} \rceil-1, & \text{if } n \equiv 0,1 ~(mod~4) \\
\lceil \frac {n}{2} \rceil, & \text{if } n \equiv 2,3 ~(mod~4)
\end{cases}$$

(ii) If $C^\sigma_n$ is an unbalanced cycle, then
$$p^{+}(C^\sigma_n) =
\begin{cases}
 \lceil \frac {n}{2} \rceil-1, & \text{if } n \equiv 2,3 ~(mod~4) \\
\lceil \frac {n}{2} \rceil, & \text{if } n \equiv 0,1 ~(mod~4)
\end{cases}$$
\end{lemma}

We have the following subsequent corollary.
\begin{corollary}\label{Corollary1}
Let $C^\sigma$ be an unbalanced signed cycle, then $p^+ (C^\sigma) \geq 2$.
\end{corollary}

Yu et al. \cite{Yu2014} calculated the positive inertia of signed path $P^\sigma_n$.
\begin{lemma}\label{Lemma3a}
Let $P^\sigma_n$ be a signed path of order $n$, then $p^{+}(P^\sigma_n)=\lfloor \frac{n}{2} \rfloor$.
\end{lemma}

Yu et al. \cite[Theorem 3.1]{Yu2016} characterized all connected signed graph $G^\sigma$ with $p^{+}(G^\sigma)=1$.
\begin{lemma}\label{Lemma4}
Let $G^\sigma$ be a connected signed graph. Then $p^{+}(G^\sigma)=1$ if and only if $G^\sigma$ is a balanced complete multipartite signed graph.
\end{lemma}

The length of a signed path (or simply a path) $P^\sigma$ refers to the total number of edges in $P^\sigma$. For any two vertices $y$ and $z$, the distance between them, denoted by $d(y,z)$, is defined as the length of the shortest path between $y$ and $z$.

\begin{lemma}\label{Lemma04}
Let $G^\sigma$ be a connected signed graph with girth $g_r$ and let $C^\sigma$ be a shortest cycle in $G^\sigma$. If $y, y' \in V(C^\sigma)$ and there exists a path $P^\sigma$ of length $k$ from $y$ to $y'$ satisfying $(V(P^\sigma) \setminus \{y, y'\}) \cap V(C^\sigma) = \emptyset$, then
\[
 \left\lceil \frac{g_r}{2} \right\rceil \leq k.
\]
\end{lemma}
\begin{proof}
Remember that $C^\sigma$ have two different paths from $y$ to $y'$. The total length of both the 
paths equals the length of $C^\sigma$, i.e., $g_r$. Thus, the shorter length of these paths is at 
most $\left\lfloor \frac{g}{2} \right\rfloor$. The shorter of these two paths from $y$ to $y'$, followed 
by the path $P^\sigma$, forms a cycle of length at most $\left\lfloor \frac{g_r}{2} \right\rfloor + k$. 
Therefore, we obtain:
\[
\left\lfloor \frac{g_r}{2} \right\rfloor + k \geq g_r,
\]
and consequently
\[
\left\lceil \frac{g_r}{2}\right\rceil \leq k.
\]
This concludes the proof.
\end{proof}

Let $G^{\sigma}[H^{\sigma}]$ be a subgraph of $G^{\sigma}$ induced by $H^{\sigma}$, and 
let $x$ be a vertex outside $H^{\sigma}$. We denote the distance between $x$ and 
$G^{\sigma}[H^{\sigma}]$ as: $d(x, G^{\sigma}[H^{\sigma}]) = \min\{ d(x, y) \mid y \in H^{\sigma} \}$. Next, 
we define $N_j(G^{\sigma}[H^{\sigma}])$ as: 
$$N_j(G^{\sigma}[H^{\sigma}]) = \{ x \in V(G^{\sigma}) \setminus H^{\sigma} \mid d(x, G^{\sigma}[H^{\sigma}]) = j, \quad j = 1, 2, \dots, n \}.$$ 
The number of vertices in $N_j(G^{\sigma}[H^{\sigma}])$ is denoted by $|N_j(G^{\sigma}[H^{\sigma}])|$.

\begin{lemma}\label{Lemma5}
Let $G^{\sigma}$ be a connected signed graph with girth $g_r$, and let $C^{\sigma}$ be a shortest cycle in $G^{\sigma}$. If $p^+(C^{\sigma}) = p^+(G^{\sigma})$, then $N_j(C^{\sigma}) = \emptyset$ for $j \geq 2$.
\end{lemma}

\begin{proof}
Indeed, we only need to establish that $N_2(C^{\sigma}) = \emptyset$. Suppose, for the 
sake of contradiction, that $N_2(C^{\sigma}) \neq \emptyset$. Let $y' \in N_2(C^{\sigma})$ 
and $y' \sim y \in N_1(C^{\sigma})$. Then, $y'$ is a pendant vertex of $G^{\sigma}[V(C^{\sigma}) \cup \{y, y'\}]$. 
Consequently,
\[
p^+(G^{\sigma}) \geq p^+(G^{\sigma}[V(C^{\sigma}) \cup \{y, y'\}]) = p^+(C^{\sigma}) + 1 > p^+(C^{\sigma}),
\]
by Lemmas \ref{Lemma01} and \ref{Lemma001}, which contradicts the assumption that
$p^+(C^{\sigma}) = p^+(G^{\sigma})$. Hence, $N_2(C^{\sigma}) = \emptyset$, and 
therefore, $N_j(C^{\sigma}) = \emptyset$ for $j \geq 3$. This completes the proof.
\end{proof}

Smith \cite{Smith1977} showed the following characterization. 
\begin{lemma}\label{Lemma05}
A graph has precisely one positive eigenvalue if and only if the set of its non-isolated vertices 
constitute a complete multipartite graph.
\end{lemma}

\section{Signed graphs with positive inertia index of 
$\lceil \frac {g_{r}}{2} \rceil-1$ and $\lceil \frac {g_{r}}{2} \rceil$}\label{sec3}

\begin{proof}[\tb {Proof of Theorem \ref{Theorem1}}]
Assume that $C^\sigma_{g_r}$ is a cycle in $G^\sigma$ with shortest length of $g_{r}$. Then by 
Lemma \ref{Lemma3} we have, $p^{+}(C^\sigma_{g_r})=\lceil \frac {g_r}{2} \rceil-1$ if 
$C^\sigma_{g_{r}}$ is a balanced signed cycle with $g_{r} \equiv 0,1 \pmod{4}$ or an unbalanced signed 
cycle with $g_{r} \equiv 2,3 \pmod{4}$ and $p^{+}(C^\sigma_{g_r})=\lceil \frac {g_r}{2} \rceil$ if 
$C^\sigma_{g_{r}}$ is a balanced signed cycle with $g_{r} \equiv 2,3 \pmod{4}$ or an unbalanced signed 
cycle with $g_{r} \equiv 0,1 \pmod{4}$. Hence, 
$p^{+}(G^\sigma) \geq p^{+}(C^\sigma_{g_r}) \geq \lceil \frac {g_r}{2} \rceil-1$ by Lemma \ref{Lemma01}, 
as desired. Next, we address the cases where equality holds in the following:\\

Clearly, by Lemma \ref{Lemma3}, we have $p^{+}(G^\sigma)=\lceil \frac {g_r}{2} \rceil-1$ if $G^\sigma$ is a balanced signed cycle 
$C^\sigma_{g_r}$ satisfying $g_{r} \equiv 0,1 \pmod{4}$ or an unbalanced signed cycle $C^\sigma_{g_r}$ 
satisfying $g_{r} \equiv 2,3 \pmod{4}$. Assume that $G^\sigma$ is a balanced 
complete multipartite signed graph $K^\sigma_{n_1,n_2, \cdots , n_t}$ with grith $g_r$. Then,
$p^{+}(K^\sigma_{n_1,n_2, \cdots , n_t})=1$, by Lemma \ref{Lemma4}. Remember that 
$g_r(K^\sigma_{n_1,n_2, \cdots , n_t})=4$, if $t=2$ and $g_r(K^\sigma_{n_1,n_2, \cdots , n_t})=3$, 
if $t \geq 3$. Thus, we have $p^{+}(G^\sigma)=1=\lceil \frac {g_r}{2} \rceil-1$.\\

Next, we prove the necessity.\\
Let $G^\sigma$ be a connected signed graph with girth $g_r$ and $p^{+}(G^\sigma)=\lceil \frac {g_r}{2} \rceil-1$. 
Then, $p^+(C^\sigma) = p^{+}(G^\sigma)=\lceil \frac {g_r}{2} \rceil-1$, which implies that $g_r \equiv 0,1 \pmod{4}$ 
or $g_r \equiv 2,3 \pmod{4}$ by Lemma \ref{Lemma3}. If $G^\sigma$ is a signed cycle $C^\sigma_{g_r}$, then 
$G^\sigma$ is either a balanced signed cycle satisfying $g_r \equiv 0,1 \pmod{4}$ or an unbalanced signed cycle 
satisfying $g_{r} \equiv 2,3 \pmod{4}$, by Lemma \ref{Lemma3}, as desired. Assume that $G^\sigma$ is not a signed 
cycle and let $y \in N_1(C^\sigma) \neq \emptyset$. Since $p^{+}(C^\sigma) = p^{+}(G^\sigma)$ by Lemma \ref{Lemma5}, 
it follows that $N_j(C^\sigma) = \emptyset$ for all $j \geq 2$.

If $y$ is adjacent to exactly one vertex $z$ of $V(C^\sigma)$, then
$$p^+(G^\sigma[V(C^\sigma) \cup \{y\}]) = p^+(G^\sigma[V(C^\sigma) \setminus \{z\}]) + 1 = 
\{\lceil \frac{g_r}{2}\rceil -1\} + 1 = \lceil \frac{g_r}{2}\rceil,$$
or
$$p^+(G^\sigma[V(C^\sigma) \cup \{y\}]) = p^+(G^\sigma[V(C^\sigma) \setminus \{z\}]) + 1 = 
\lceil \frac{g_r}{2}\rceil + 1 = \lceil \frac{g_r}{2}\rceil+1,$$
by Lemmas \ref{Lemma001} and \ref{Lemma3}, and so
$$p^+(G^\sigma) \geq p^+(G^\sigma[V(C^\sigma) \cup \{y\}]) = \lceil \frac{g_r}{2}\rceil,$$
or
$$p^+(G^\sigma) \geq p^+(G^\sigma[V(C^\sigma) \cup \{y\}]) = \lceil \frac{g_r}{2}\rceil+1,$$
by Lemma \ref{Lemma01}, a contradiction.

Thus, the vertex $y$ is at least neighboring two distinct vertices $z, z' \in V(C^\sigma)$, which suggests that 
there exists a path $P^\sigma$ of given length $2$ from $z$ to $z'$ which satisfies 
$(V(P) \setminus \{z, z'\}) \cap V(C^\sigma) = \emptyset$.

Consequently, \(\lceil \frac{g_r}{2} \rceil \leq 2\) by Lemma \ref{Lemma04}, meaning that $g_r = 3$ or $4$, 
and therefore $p^+(G^\sigma) = \lceil \frac{g}{2} \rceil - 1 = 1$. Hence, $G^\sigma$ is a complete 
multipartite signed graph ($g_r = 3$ or $g_r = 4$) according to Lemma \ref{Lemma05}.\

 Next, we have to show that $G^\sigma$ is balanced. On contrary, assume that $G^\sigma$ is unbalanced. Then, there 
 must exist an unbalanced cycle in $G^\sigma$ as an induced subgraph. By Corollary \ref{Corollary1} and Lemma 
 \ref{Lemma01}, $p^+(G^\sigma) \geq 2$, which arises contradiction. Thus, $G^\sigma$ is a balanced complete multipartite 
 signed graph if $p^+(G^\sigma)=1$.
\end{proof}

The proof of Theorem \ref{Theorem2} can be directly followed from Theorem \ref{Theorem1}, Lemmas \ref{Lemma01} 
and \ref{Lemma3} by excluding the two cases of Theorem \ref{Theorem1}, where the equality holds.\\

In what follows, we characterize connected signed graphs $G^\sigma$ with girth $g_r$ and 
$p^+(G^\sigma) = \lceil \frac{g}{2} \rceil$. The following lemma \ref{Lemma6} establishes the structure of 
a canonical unicyclic graph $G^\sigma$ with girth $g_r$ satisfying $p^+(G^\sigma) = \lceil \frac{g_r}{2} \rceil$.

\begin{lemma}\label{Lemma6}
 Let $G^\sigma$ be a canonical unicyclic signed graph with girth $g_r$. Then $p^+(G^\sigma)=\lceil \frac{g_r}{2}\rceil$ 
 if and only if:\\
(i) $G^\sigma=C^\sigma_{g_{r}}$ is a balanced signed cycle satisfying $g_{r} \equiv 2,3 \pmod{4}$ or an unbalanced 
signed cycle satisfying $g_{r} \equiv 0,1 \pmod{4}$;\\
(ii) $G^\sigma$ be a graph where either exactly one vertex is attached to a pendant star, or multiple vertices are 
each attached to a pendant star such that all paths having odd number of vertices between any two consecutive pendant 
stars when $g_r=1,3 \pmod4$;\\
(iii) $G^\sigma$ be a graph where either exactly one vertex is attached to a pendant star, or multiple vertices are 
each attached to a pendant star such that there is exactly one path with an even number of vertices (including zero) 
between any two consecutive pendant stars when $g_r=0,2 \pmod4$.

\end{lemma}
\begin{proof}
If $G^\sigma=C^\sigma_{g_{r}}$, then $p^+(G^\sigma)=\lceil \frac{g_r}{2}\rceil$ if and only if 
$G^\sigma$ is a balanced signed cycle satisfying $g_{r} \equiv 2,3\pmod{4}$ or an unbalanced 
signed cycle satisfying $g_{r} \equiv 0,1\pmod{4}$ by Lemma \ref{Lemma3}. Hence, we now consider 
the case when $G^\sigma$ is not a signed cycle.

Let $C^\sigma$ be the unique signed cycle of $G^\sigma$ with $g_r$ vertices. Assume that $t \geq 1$ 
pendant stars are attached to $C^\sigma$, and let $P^\sigma_{n_1}, P^\sigma_{n_2}, \dots, P^\sigma_{n_t}$ 
be the paths which can be obtained after removing $t$ pendant stars from $G^\sigma$. Then we have:
\[
g_r = t + n_1 + n_2 + \dots + n_t.\]

Since $G^\sigma$ contains $t$ pendant stars, it follows from Lemma \ref{Lemma001} that:
\[
t + p^+(P^\sigma_{n_1}) + p^+(P^\sigma_{n_2}) + \dots + p^+(P^\sigma_{n_t}) = p^+(G^\sigma).
\]
On the other hand, $p^+(G^\sigma) = \lceil \frac{g_r}{2}\rceil$, thus we have
\[
t + p^+(P^\sigma_{n_1}) + p^+(P^\sigma_{n_2}) + \dots + p^+(P^\sigma_{n_t}) = \lceil \frac{g_r}{2}\rceil.
\]

Next, we divide the discussion into the following two cases:\\

\textbf{Case 1:} $g_r \equiv 1, 3 \pmod{4}$.\\
In this case, we have $p^+(G^\sigma) = \lceil \frac{g_r}{2}\rceil$ if and only if
\[
t + p^+(P^\sigma_{n_1}) + p^+(P^\sigma_{n_2}) + \dots + p^+(P^\sigma_{n_t}) = 
\bigg \lceil \frac{t + n_1 + \dots + n_t}{2} \bigg \rceil = \frac{t + n_1 + \dots + n_t}{2}+1.
\]
This implies that:
\[
t - 1 = \big(n_1 - 2p^+(P^\sigma_{n_1})\big) + \big(n_2 - 2p^+(P^\sigma_{n_2})\big) + \dots + 
\big(n_t - 2p^+(P^\sigma_{n_t})\big).
\]
Since $p^+(P^\sigma_n) = \lfloor n/2 \rfloor$ by Lemma \ref{Lemma3a}, we have $n_j-2p^+(P^\sigma_{n_j})=0$ if 
$n_j=$ even (where, $1 \leq n_j \leq t$) and $n_j-2p^+(P^\sigma_{n_j})=1$ if $n_j=$ odd (where, $1 \leq n_j \leq t$). 
Thus, we conclude that either $t = 1$ or exactly one path has an even number of vertices (including zero) 
between two consecutive pendant stars, while the others have an odd number of vertices, as required.\\

\textbf{Case 2:} $g_r \equiv 0, 2 \pmod{4}$.\\
In this case, we have $p^+(G^\sigma) = \lceil \frac{g_r}{2}\rceil$ if and only if
\[
t + p^+(P^\sigma_{n_1}) + p^+(P^\sigma_{n_2}) + \dots + p^+(P^\sigma_{n_t}) = 
\bigg \lceil \frac{t + n_1 + \dots + n_t}{2} \bigg \rceil = \frac{t + n_1 + \dots + n_t}{2}.
\]
This implies that:
\[
t = \big(n_1 - 2p^+(P^\sigma_{n_1})\big) + \big(n_2 - 2p^+(P^\sigma_{n_2})\big) + \dots + 
\big(n_t - 2p^+(P^\sigma_{n_t})\big).
\]
By similar reasoning as in Case 1, we have $n_j-2p^+(P^\sigma_{n_j})=0$ if $n_j=$ even (where, $1 \leq n_j \leq t$) 
and $n_j-2p^+(P^\sigma_{n_j})=1$ if $n_j=$ odd (where, $1 \leq n_j \leq t$). Therefore, we conclude that either 
$t = 1$ or all paths must have an odd number of vertices between any two pendant stars, as required.
\end{proof}

\begin{lemma}\label{Lemma7}
Let \( G^\sigma \) be a connected signed graph with girth \( g_r \) which is not a canonical unicyclic signed graph. 
Then, \( p^+(G^\sigma) = \lceil \frac{g}{2} \rceil \) 
if and only if \( G^\sigma \) is isomorphic
to one of the following graphs.\\
(i) A signed graph $G^\sigma$ obtained from a balanced signed cycle $C^\sigma_{g_r}$ with $g_r \equiv 0,1 \pmod 4$ and 
signed star $S^\sigma_{t+1}$ by joining a vertex of $C^\sigma_{g_r}$ to the center of $S^\sigma_{t+1}$, 
where $t \geq 1$.\\
(ii) A signed graph $G^\sigma$ obtained from an unbalanced signed cycle $C^\sigma_{g_r}$ with $g_r \equiv 2,3 \pmod 4$ 
and signed star $S^\sigma_{t+1}$ by joining a vertex of $C^\sigma_{g_r}$ to the center of $S^\sigma_{t+1}$, 
where $t \geq 1$.\\
(iii) Bicyclic signed graphs \( B^\sigma(5, 4, 5) \) and \( B^\sigma(5, 5, 5) \) with balanced cycles or a signed 
bicyclic graph \( B^\sigma(5, 3, 5) \) whose both cycles are negative, depicted in Figure \ref{Bicyclic}.
\end{lemma}

\begin{proof}
The bicyclic signed graph \( B^\sigma(5, 5, 5) \) has girth $8$ and \( B^\sigma(5, 4, 5) \) has girth $7$. By 
performing a simple operation as given in Figure \ref{Bicyclic}, we have
\[
p^+(B^\sigma(5, 5, 5)) = p^+(B^\sigma(5, 4, 5)) = 4 = \left\lceil \frac{g_r}{2} \right\rceil.
\]
Similarly, \( B^\sigma(5, 3, 5) \) has girth $6$. If both the cycles are negative, we have 
$p^+(B^\sigma(5, 3, 5)) = 3 = \lceil \frac{g_r}{2} \rceil$, see Figure \ref{Bicyclic}.\\

Next, assume that \( G^\sigma \) is a signed graph obtained from a balanced signed cycle \( C^\sigma_{g_r} \) 
with $g_r \equiv 0,1 \pmod 4$ and $S^\sigma_{t+1}$ by joining a vertex of \( C^\sigma_{g_r} \) to the 
center of $S^\sigma_{t+1}$. By removing a pendant vertex along with a center vertex of 
$S^\sigma_{t+1}$ from \( G^\sigma \), we have
\[
p^+(G^\sigma) = p^+(C^\sigma_{g_r}) + 1 = \left(\left\lceil \frac{g}{2} \right\rceil - 1\right) + 1 = 
\left\lceil \frac{g}{2} \right\rceil ,
\]
by Lemmas \ref{Lemma001} and \ref{Lemma3}.

If \( G^\sigma \) is a signed graph obtained from an unbalanced signed cycle \( C^\sigma_{g_r} \) with 
$g_r \equiv 2,3 \pmod 4$ and $S^\sigma_{t+1}$ by joining a vertex of \( C^\sigma_{g_r} \) to 
the center of $S^\sigma_{t+1}$. By removing a pendant vertex along with a center vertex of 
$S^\sigma_{t+1}$ from \( G^\sigma \), we have
\[
p^+(G^\sigma) = p^+(C^\sigma_{g_r}) + 1 = \left(\left\lceil \frac{g}{2} \right\rceil - 1\right) + 1 = \left\lceil \frac{g}{2} \right\rceil ,
\]
by Lemmas \ref{Lemma001} and \ref{Lemma3}.

The necessity part of the proof can be directly followed from the case of ordinary graphs, as established 
in \cite[Lemma 3.3]{DuanYang2024}, to the context of signed graphs.
\end{proof}

\begin{figure}[htbp!]
     \centering
  \includegraphics[width=12cm]{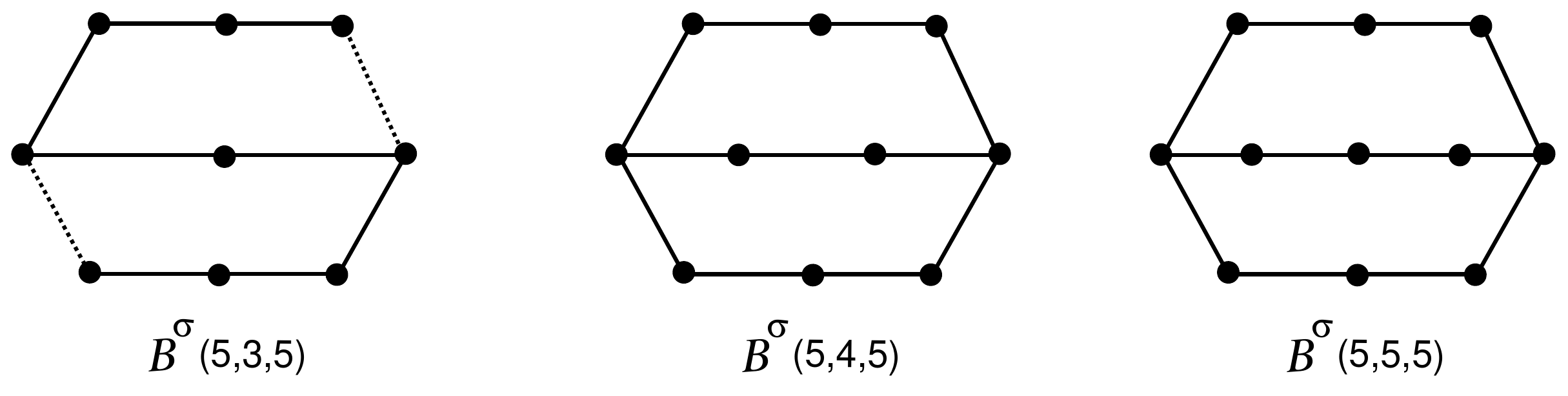}\\
  \caption{Bicyclic signed graphs \( B^\sigma(5, 3, 5) \), \( B^\sigma(5, 5, 5) \) and \( B^\sigma(5, 4, 5) \), 
  where the doted lines represent the negatived edges.}\label{Bicyclic}
\end{figure}

\begin{proof}[\tb {Proof of Theorem \ref{Theorem3}}]
By combining the results of Lemmas \ref{Lemma6} and \ref{Lemma7}, we establish the characterization of the 
connected signed graph \( G^\sigma \) with girth \( g_r \) and \( p^+(G^\sigma) = \lceil \frac{g_r}{2} \rceil \), 
and hence the proof of Theorem \ref{Theorem3} is completed.
\end{proof}


\end{document}